\documentclass[11pt]{article}
\usepackage{amsmath,amsfonts,amssymb,amsthm}
\usepackage[margin=1.25in]{geometry}
\usepackage[round]{natbib}

\DeclareRobustCommand{\sortandprefix}[3]{#2} 

\newcommand{\equalcontrib}{\footnotemark[\arabic{footnote}]}

\usepackage[colorlinks=true,citecolor=black,urlcolor=black,linkcolor=black]{hyperref}
\usepackage{authblk}

\newcommand{\E}{\mathbb{E}}
\newcommand{\calP}{\mathcal{P}}
\newcommand{\Peff}{\calP_{\mathrm{eff}}}

\newcommand{\calS}{\mathcal{S}}
\newcommand{\calE}{\mathcal{E}}
\newcommand{\calF}{\mathcal{F}}
\renewcommand{\Re}{\mathbb{R}}
\newcommand{\KL}{D_{\mathrm{KL}}}
\newcommand{\ind}{\mathbf{1}}
\DeclareMathOperator{\Cauchy}{Cauchy}
\newcommand{\given}{ \ \middle| \ }
\DeclareMathOperator{\Par}{Pareto}

\newtheorem{theorem}{Theorem}
\newtheorem{lemma}{Lemma}
\newtheorem{corollary}{Corollary}
\newtheorem{example}{Example}
\newtheorem{remark}{Remark}

\title{Rao-Blackwellized e-variables\thanks{Published in \textit{Biometrika}.}}
\author[1,3]{Dante de Roos\thanks{Equal contribution.}}
\author[2]{Ben Chugg\protect\equalcontrib}
\author[1,3]{Peter Gr\"{u}nwald}
\author[4]{Aaditya Ramdas}
{
\affil[1]{\small Centrum Wiskunde \& Informatica, The Netherlands}
\affil[2]{Carnegie Mellon University, USA}
\affil[3]{Leiden University, The Netherlands}
\affil[4]{Stanford University, USA}
}
\date{October 2026}

\begin{document}

\maketitle

\begin{abstract}
We show that for any concave utility, the expected utility of an e-variable can only increase after conditioning on a sufficient statistic, echoing the classic Rao-Blackwell theorem. The simplest form of the result has a straightforward proof using Jensen's inequality, but we prove more sophisticated versions.  We provide several applications of this insight, including a derivation of the log-optimal e-variable for testing the variance in a linear regression model.
\end{abstract}


\section{Introduction}

The Rao-Blackwell theorem~\citep{radhakrishna1945information,blackwell1947conditional} is a famous result in statistical decision theory which states that, for any convex loss function, the expected loss of an estimator can only decrease after conditioning on a sufficient statistic. The act of conditioning a random variable on a sufficient statistic has thus come to be known as ``Rao-Blackwellization.''

This short article makes the observation that a Rao-Blackwell type result exists for e-variables and related objects. An e-variable for a set of distributions $\calP$ is a $[0,\infty]$-valued random variable with expectation at most $1$ under every $P\in \calP$. E-variables and various extensions, such as e-processes and asymptotic and compound e-variables, have emerged as useful tools in mathematical statistics and have garnered significant attention over the past five years~\citep{ramdas2025hypothesis,grunwald2024safe}.

In the context of hypothesis testing, one takes $\calP$ to be the null hypothesis. Given an e-variable $E$, a level-$\alpha$ test is defined by rejecting the null if $E\geq 1/\alpha$, since type-I error control is guaranteed by Markov's inequality: $P(E\geq 1/\alpha)\leq \E_P[E]\alpha\leq \alpha$ for all $P\in\calP$. The goal is thus to design e-variables for $\calP$ which are large under the alternative hypothesis.

In fact, one can often construct e-variables which tend to become \emph{exponentially} large under the alternative. This suggests maximizing its exponent via maximizing the expectation of $\log(E)$ under the alternative. This is also known as ``Kelly betting'' \citep{kelly1956new}, and is equivalent to maximizing the geometric growth rate of $E$. E-variables which maximize the expected log under the alternative are called \emph{log-optimal} or \emph{growth-rate optimal} (GRO) e-variables \citep{grunwald2024safe, ramdas2023game}. Below we consider a general class of utilities, including the log-utility.

Rao-Blackwell type results have occurred in various places in the e-variable literature. In the context of group-invariance testing, \citet{perez2024statistics} show that one can first reduce the data to an invariantly sufficient statistic and then look for optimal e-variables in this smaller space. \citet{ramdas2025hypothesis} demonstrate that permutation-averaging of e-variables increases the expected log-utility when the data are exchangeable. Recently, for point nulls and alternatives, \citet{chugg2026admissibility} showed that post-hoc hypothesis tests can be likewise improved, in a result that can be viewed as a preliminary version of Theorem~\ref{thm:rb-evals}, while \citet{deRoos2025} provided a preliminary version of Theorem~\ref{thm:rb-log-util-general} in an unpublished Master's thesis. This work may be seen as a unification and generalization of these two preliminary results, equipped with additional examples, with which we aim to demonstrate that Rao-Blackwell type results for e-variables hold in great generality. Moreover, in contrast to the results presented by \citet{radhakrishna1945information, blackwell1947conditional}, we impose few to no integrability assumptions.

\section{A Rao-Blackwell theorem for e-variables}
\subsection{Preliminaries}
We consider a set of probability measures $\{P_\theta\}_{\theta\in\Theta}$ on a sample space $(\Omega, \calF)$, with $\calF$ a $\sigma$-field, and we assume $\Theta$ is partitioned into the null $\Theta_0$ and alternative $\Theta_1$. Data will be denoted by some generic random element $X \in \Omega$, and by an e-variable we mean an e-variable for the collection $ \calP =\{P_\theta\}_{\theta \in \Theta_0}$.

In what follows, by a \textit{random variable} we mean a measurable function $Y$ from $\Omega$ into the extended reals $[-\infty, \infty]$. The expectation $\E[Y]$ is said to be \textit{(well-) defined (as an element of $[-\infty, \infty]$)} if $\min(\E[Y^+], \E[Y^-]) < \infty$, where $x^+ = \max(0, x)$ and $x^- = \max(0, -x)$. This condition ensures that $\E[Y] = \E[Y^+] - \E[Y^-]$ is unambiguous as an element of $[-\infty, \infty]$. The definition for conditional expectations $\E[Y \mid \mathcal{G}]$ is analogous, and if $\E[Y]$ is defined then $\E[Y \mid \mathcal{G}]$ is also defined. For more details on generalized (conditional) expectations, we refer to \citet{shiryaev2016probability}.

A sufficient statistic for $\Theta$ is often heuristically defined as a measurable function $S:\Omega \to \calS$ such that the conditional distribution of the data $X\sim P_\theta$ given $S$ is independent of $\theta$, for all $\theta\in \Theta$. Formally, this means that for every measurable function $h$, there exists some measurable $g_h$, not depending on $\theta \in \Theta$, such that $\E_{\theta} [h(X) \mid  S] = g_h(S)$ $P_\theta$-a.s for all $\theta\in\Theta$, provided that all these conditional expectations are defined. When conditioning on a sufficient statistic, we will thus write $\E_\Theta[\cdot\mid S]$ without referring to any particular parameter $\theta\in\Theta$. Here, conditioning on $S$ should be understood as conditioning on $\sigma(S)$, the $\sigma$-field generated by $S$.

In fact, the mathematical theory of sufficiency, as developed following \citet{bahadur1954sufficiency}, concentrated on the more
general concept of sufficient $\sigma$-fields. In this article, however, we will state and prove most of our results in terms of sufficient statistics to accommodate a more gentle first reading. We do emphasize, however, that our results directly generalize to sufficient $\sigma$-fields.

Finally, a \textit{concave utility function} is a concave function $f:(0, \infty) \to \Re$ which is extended to $f:[0, \infty] \to [-\infty, \infty]$ by putting $f(0) = \lim_{x \downarrow 0}f(x)$ and $f(\infty) = \lim_{x \to \infty}f(x)$. This convention ensures that the random variable $f(Y)$ is well-defined for every nonnegative random variable $Y$.

\subsection{Main result}
\begin{theorem} \label{thm:rb-evals}
Let $S$ be a sufficient statistic for $\Theta$. Let $E$ be an e-variable for $\Theta_0$ and set $G = \E_{\Theta}[E\mid S]$.
Then $G$ is an e-variable
and $\E_{\theta}[f(G)] \geq \E_{\theta}[f(E)]$ for any concave utility function $f$ and any $\theta\in\Theta$, assuming both sides of this inequality are well-defined.
\end{theorem}
\begin{proof}
By definition of sufficiency, $\E_\theta[G] = \E_\theta[\E_\theta[E\mid S]] = \E_\theta[E]$ for any $\theta$, implying that $G$ is a bona fide e-variable: $\sup_{\theta\in\Theta_0} \E_\theta[G] = \sup_{\theta\in\Theta_0}\E_\theta[E]\leq 1$. The rest of the proof becomes embarrassingly simple after realizing that $\E_\theta[f(E)] = \E_\theta[\E_\theta[f(E) \mid S]] \leq \E_\theta[f(\E_\theta[E \mid S])] = \E_\theta[f(G)]$ follows by the tower property and Jensen's inequality under suitable integrability assumptions. Such integrability assumptions are completely redundant, however, because we show in Lemma~\ref{lem:general-jensen-ineq} in Appendix~\ref{sec:jensen} that the above reasoning also applies under the minimal assumption that $\E_\theta[f(E)]$ and $\E_\theta[f(G)]$ are well-defined elements of $[-\infty, \infty]$.
\end{proof}

\begin{remark}
\label{rem:any-variable}
    The random variable $E$ need not be an e-variable in Theorem~\ref{thm:rb-evals} for the inequality to hold, just a nonnegative random variable. This insight also allows us to Rao-Blackwellize other types of e-statistics, such as asymptotic or compound e-variables, as well as e-processes. We discuss these cases separately in Appendix~\ref{sec:extensions}.
\end{remark}
\begin{remark}
    Although Theorem~\ref{thm:rb-evals} is true for any concave utility function, its meaning and applicability rely on the existence of $\E_\theta[f(E)]$ and $\E_\theta[f(G)]$, which depends on $f$, $\Theta$ and $E$. A class of such $f$ for which these expectations are always defined is the class of concave utility functions that are bounded from above or below. This class, for instance, includes the power utilities $f(x) = x^{1-\gamma}/(1-\gamma)$, $\gamma > 1$, which have been studied by \citet{larsson2025numeraire}. The ubiquitous log-utility, on the other hand, is not a member of this class, and we will illustrate in Example~\ref{ex:cauchy} that Theorem~\ref{thm:rb-evals} is indeed not always applicable to this utility.
\end{remark}

To make things concrete, let us give a brief example of using Rao-Blackwellization to strictly improve the expected utility of a naive e-variable. A common example in the original Rao-Blackwell setting is to consider an estimator which only uses the first observation. We follow in this tradition.

\begin{example}
\label{ex:bern-naive-eval}
Let $X_1,\dots,X_n$ be i.i.d. $\mathrm{Ber}(\theta)$ with parameter $\theta \in(0,1)$.
We test the simple null $\Theta_0=\{\theta_0\}$ for some fixed $\theta_0\in(0,1)$
against some alternative $\Theta_1$.
A sufficient statistic for $\Theta$ is the total number of successes $S = \sum_{i=1}^n X_i$. The statistic $E = \exp\bigl(\lambda X_1 - \psi(\lambda)\bigr)$ where $\lambda>0$ is fixed and $\psi(\lambda) = \log \bigl( \E_{\theta_0}[\exp(\lambda X_1)] \bigr) = \log\bigl(\theta_0 e^\lambda + (1-\theta_0)\bigr)$, is an e-variable by construction, but it is rather naive as it only uses the first data point. Consider $G(S) = \E_{\Theta}[E\mid S]$, the Rao-Blackwellization of $E$. It is easy to see that $P_\theta(X_1 = 1 \mid S = k) = 1-P_\theta(X_1 = 0 \mid S =k ) = k/n$, from which it follows that
\begin{align*}
G(S)
&= e^{-\psi(\lambda)}
   \left(\frac{S}{n} e^\lambda + \bigg(1-\frac{S}{n}\bigg)\right)
 = e^{-\psi(\lambda)}\left(1 + (e^\lambda-1)\frac{S}{n}\right).
\end{align*}
This is indeed also an e-variable, because $\E_{\theta_0}[S/n] = \theta_0$. Now, let $f$ be any strictly concave function on $(0,\infty)$.
We claim that $\E_\theta[f(G)] > \E_\theta[f(E)]$ for all $\theta\in(0,1)$. By the strict concavity of $f$, for $k\in\{1,\dots,n-1\}$, $\E_\theta[f(E) \mid  S=k]<f(\E_\theta[E \mid  S=k]) = f(G(k))$, while for $k\in\{0,n\}$, $\E_\theta[f(E) \mid  S=k] = f(G(k))$. Since $P_\theta(S\neq 0,n) = 1 - \theta^n - (1-\theta)^n>0$, it follows that $\E_\theta[f(G)] > \E_\theta[f(E)]$,
demonstrating strict utility improvement.
\end{example}

An immediate consequence of Theorem~\ref{thm:rb-evals} is that if there exists an e-variable $E^*$ that maximizes $\E_\theta[f(E)]$ over the set of e-variables, then $G = \E_\Theta[E \mid S]$ is also a maximizer. Consequently, one may restrict the search for such ``$f$-optimal'' e-variables to those that are functions of $S$, which might considerably simplify the underlying optimization problem. Furthermore, if optimal e-variables are known to be essentially unique, which is the case for log-optimal (GRO) e-variables, then one may immediately conclude that an e-variable cannot possibly be optimal if it cannot be written as a function of $S$. We provide examples of these insights in Section~\ref{sec:applications} below.

We must emphasize that carelessly applying Rao-Blackwellization is not guaranteed to yield optimal e-variables. In Example~\ref{ex:bern-naive-eval}, for instance, the log-optimal e-variable for testing the point null $\{\theta_0\}$ against the point alternative $\{\theta_1\}$ is the likelihood ratio $\prod_{i=1}^n p_{\theta_1}(X_i)/p_{\theta_0}(X_i)$, which is not equal to $G(S)$ for any choice of $\lambda$. This likelihood ratio is a function of the sufficient statistic, and therefore remains unaffected by Rao-Blackwellization.

\subsection{The log-utility}
The requirement that $\E_\theta[f(E)]$ and $\E_\theta[f(G)]$ are well-defined in Theorem~\ref{thm:rb-evals} is not always met. But Rao-Blackwellization may in such cases still lead to sensible results. In this section we focus on the log-utility in particular, demonstrating how in this case we can weaken the assumptions on the existence of $\E_\theta[\log(E)]$ and $\E_\theta[\log(G)]$ in Theorem~\ref{thm:rb-evals}.

We begin by giving sufficient conditions under which $\E_\theta[\log(E)]$ is defined for every e-variable $E$, thereby ensuring that Theorem~\ref{thm:rb-evals} is meaningful. We then present an example where these conditions are not satisfied, and $\E_\theta[\log(E)]$ does indeed not exist. This example motivates Theorem~\ref{thm:rb-log-util-general}, a generalization of Theorem~\ref{thm:rb-evals} for the log-utility, that allows us to compare $E$ with its Rao-Blackwellization $G$, and conclude that $G$ is an improvement over $E$.

For the following lemma, we need two additional definitions. First, the \textit{KL-divergence} \citep{kullback1951information} $\KL(\nu \| \mu)$ between two finite measures $\nu$ and $\mu$ is defined as $\int \log (d\nu/d\mu) d\nu$ if $\nu$ is absolutely continuous with respect to $\mu$ and $+\infty$ otherwise. Jensen's inequality ensures that the KL-divergence is always well-defined as an element of $(-\infty, \infty]$. Second, the \textit{effective null hypothesis} \citep{larsson2025numeraire} $\Peff$ is the set $\{P \in M_+ \mid  \E_P[E] \leq 1 \textnormal{ for all } E \in \calE\}$ where $M_+$ denotes the set of (nonnegative) measures on $(\Omega, \calF)$, and $\calE$ the set of e-variables for $\calP$. The effective null hypothesis, consisting only of (sub-) probability measures, is the largest extension of the null hypothesis $\calP$ for which $\calE$ is still the class of all e-variables. It includes any ($\sigma$-) convex combination of elements of the null as well as all Bayes marginals, when defined.
\begin{lemma}
\label{lem:log-utility-defined}
Let $\theta \in \Theta$. If $\inf_{P \in \Peff}\KL(P_\theta \|  P) < \infty$, then $\E_{\theta}[\log(E)]$ is defined for every e-variable $E$.
\end{lemma}
The proof is deferred to Appendix~\ref{proof:log-utility-defined}.
\begin{remark}
    The condition $\inf_{P \in \Peff}\KL(P_\theta \| P) < \infty$ can already be verified by constructing a mixture of elements of the null $P'$, and checking whether $\KL(P_\theta \|  P') < \infty$. In most practical situations, this will be the viable strategy, as precisely identifying the effective null is often too difficult. This strategy is not always successful, yet the current formulation of Lemma~\ref{lem:log-utility-defined} may still cover such cases; see, for instance, \citet[Example~4.9]{larsson2025numeraire}.
\end{remark}

In the following example we show that $\E_\theta[\log(E)]$ cannot be guaranteed to be defined if the conditions of Lemma~\ref{lem:log-utility-defined} are not satisfied.

\begin{example}
\label{ex:cauchy}
    Consider the simple versus simple testing problem $X \sim P_{\theta_0}=N(0,1)$ versus $X \sim P_{\theta_1} = \Cauchy(0,1)$. Here, $\KL(P_{\theta_1} \| P_{\theta_0}) = \infty$, and, because $\Peff$ consists of all (sub-) probability measures $P$, such that $P(A) \leq P_{\theta_0}(A)$ for all $A \in \calF$ \citep[Example 3.12]{larsson2025numeraire}, $\inf_{P \in \Peff} \KL(P_{\theta_1} \| P) = \KL(P_{\theta_1} \| P_{\theta_0}) = \infty$. Indeed, the conditions of Lemma~\ref{lem:log-utility-defined} are not met.

    Both $P_{\theta_0}$ and $P_{\theta_1}$ are symmetric about zero, so the statistic $S = |X|$ is sufficient for $\Theta = \{\theta_0, \theta_1\}$, and $\E_\Theta[h(X) \mid  S] = (h(X) + h(-X))/2$. The statistic $E = \exp(X-1/2)$ is an e-variable, but $\E_{\theta_1}[\log(E)] = \E_{\theta_1}[X -1/2]$ is not defined, since $\E_{\theta_1}[X^+] = \E_{\theta_1}[X^-] = \infty$. On the other hand, $E$'s Rao-Blackwellization $G = \E_\Theta[\exp(X-1/2) \mid  S] = (\exp(X -1/2) + \exp(-X -1/2))/2$ is bounded from below, so $\E_{\theta_1}[\log(G)]$ is defined; in fact, $\E_{\theta_1}[\log(G)] = \infty$. Hence, $E$ and $G$ are not comparable in the sense of Theorem~\ref{thm:rb-evals}.
\end{example}
\begin{remark}
    The practical relevance of Example~\ref{ex:cauchy} might, at first glance, seem limited. Exponential-style e-variables are, however, commonly used in practice for light-tailed nulls \citep{howard2020time}. For alternatives where $X$ has an infinite mean in both tails the expected log-utility of these e-variables is undefined, and Example~\ref{ex:cauchy} should therefore be viewed as a simple illustration of a general, practically relevant phenomenon.
\end{remark}

Faced with Example~\ref{ex:cauchy}, we next give a stronger result than Theorem~\ref{thm:rb-evals} which enables us to compare $G$ and $E$ even when $\E_\theta[\log(E)]$ does not exist.
Following the literature \citep{lardy2024reverse, larsson2025numeraire}, we formulate the improvement of $G$ over $E$ in terms of their ratio. We use the conventions $0/0=0$ and $\infty/\infty=1$.  The proof can be found in Appendix~\ref{proof:rb-log-util-general}.

\begin{theorem} \label{thm:rb-log-util-general}
Let $E$ be an e-variable for $\Theta_0$, and $S$ a sufficient statistic for $\Theta$. Set $G = \E_\Theta[E\mid S]$. Then for all $\theta\in\Theta$,  $\E_\theta[E/G]\leq 1$, and $\E_\theta[\log(E/G)] \leq 0$.
\end{theorem}

Remark~\ref{rem:any-variable} also applies to Theorem~\ref{thm:rb-log-util-general}. That is, $E$ need not be an e-variable; just a nonnegative random variable. The above result overcomes the pathology in Example~\ref{ex:cauchy}.
\begin{example}
    Let us continue with Example~\ref{ex:cauchy}, where $E/G = 2(1 + \exp(-2X))^{-1}$. One easily infers
    \begin{align*}
        \E_{\theta_1}[\log(E/G)] \leq \int_{-\infty}^0 \frac{\log(2) + 2x}{\pi(1 + x^2)}d x + \log(2) = -\infty.
    \end{align*}
    Thus, Theorem~\ref{thm:rb-log-util-general} still allows $E$ and $G$ to be compared, even though Theorem~\ref{thm:rb-evals} does not.
\end{example}

\section{Applications}
\label{sec:applications}
\subsection{Setup}
In this section, we present two applications of Theorems~\ref{thm:rb-evals} and~\ref{thm:rb-log-util-general} by (i) demonstrating how to recover GRO (growth-rate optimal) or GROW (growth-rate optimal in the worst-case) e-variables when testing the
variance in a linear regression model, and (ii) by demonstrating that permutation-averaging increases expected utility when both the null and the alternative consist of exchangeable distributions. While we have discussed these optimality criteria informally, let us define them here.

Let $\calE$ be the set of e-variables for the null hypothesis $\Theta_0$. For a point alternative $\Theta_1 = \{\theta_1\}$, the GRO e-variable is that which maximizes the expected log-utility under the alternative: $\sup_{E \in \calE} \E_{\theta_1}[\log(E)]$. When the alternative is composite one needs to decide how to compare e-variables under parameters $\theta_1\in\Theta_1$. One option is to maximize log-utility in the worst case. This leads to the GROW e-variable, which attains the worst case log-utility: $\sup_{E \in \calE}\inf_{\theta_1 \in \Theta_1}\E_{\theta_1}[\log(E)]$.

It is necessary for these criteria to be well-defined that $\E_{\theta_1}[\log(E)]$ be defined for all $E \in \calE$ and $\theta_1 \in \Theta_1$, which can be verified using Lemma~\ref{lem:log-utility-defined}. If these expectations are not all defined, then for simple alternative hypotheses the GRO criterion can be replaced by the relative log-optimality criterion, introduced by \citet{lardy2024reverse, larsson2025numeraire}, to which Theorem~\ref{thm:rb-log-util-general} specifically relates.
\subsection{GRO(W) e-variable for testing the variance in a linear model} \label{sec:linear-reg}
Consider a linear regression model given by $Y \sim N(X\beta, \sigma^2 I_n)$, with parameters $\beta \in \Re^k$ and $\sigma^2 \in (0, \infty)$, and design matrix $X \in \Re^{n \times k}$ which we assume to be known and of full rank. We are interested in testing the null hypothesis saying that $\sigma^2=\sigma^2_0$ against the alternative hypothesis that $\sigma^2 = \sigma^2_1$. Put formally,
\begin{align} \label{eq:lin-reg-testing-problem}
    H_0:Y \sim N(X\beta, \sigma_0^2I_n), \, \beta \in \Re^k, \quad H_1:Y\sim N(X\beta, \sigma_1^2I_n), \, \beta \in \Re^k.
\end{align}
Here, $\Theta_0 = \{(\beta, \sigma_0^2) \mid \beta \in \Re^k \}$ and $\Theta_1 = \{(\beta, \sigma_1^2) \mid \beta \in \Re^k\}$. Because the KL-divergence between any two multivariate normal distributions is finite, it follows by Lemma~\ref{lem:log-utility-defined} that the GRO criterion for a point alternative $\theta_1 \in \Theta_1$ and the GROW criterion for the composite alternative $\Theta_1$, are well-defined.

It is a well-known fact that the maximum likelihood estimator $(\hat \beta, \hat\sigma)$ is sufficient for the full linear regression model, and
\begin{align*}
    (\hat\beta, \hat\sigma^2) \sim N(\beta , \sigma^2(X^\top X)^{-1}) \otimes \frac{\sigma^2}{n}\chi^2_{n-k},
\end{align*}
where $a\chi^2_{n-k}$ denotes the distribution of $aZ$ when $Z \sim \chi^2_{n-k}$.

Now, by Theorem~\ref{thm:rb-evals}, it follows that the GRO(W) e-variable for testing \eqref{eq:lin-reg-testing-problem} can be written as a function of $(\hat\beta, \hat\sigma)$, so the GRO(W) e-variable for the reduced testing problem
\begin{equation} \label{eq:lin-reg-reduced-testing-problem}
\begin{split}
    &H_0': (\hat\beta, \hat\sigma^2) \sim N(\beta , \sigma_0^2(X^\top X)^{-1}) \otimes \frac{\sigma_0^2}{n}\chi^2_{n-k}, \, \beta \in \Re^k, \\
    &H_1':(\hat\beta, \hat\sigma^2) \sim N(\beta , \sigma_1^2(X^\top X)^{-1}) \otimes \frac{\sigma_1^2}{n}\chi^2_{n-k}, \,  \beta \in \Re^k,
\end{split}
\end{equation}
will be the GRO(W) e-variable for \eqref{eq:lin-reg-testing-problem}. Now, denote the class of e-variables for $H_0'$ by $\calE'$. Furthermore, write $P_{\beta,\sigma^2} = N(\beta, \sigma^2(X^\top X)^{-1})$, $Q_{\sigma^2} = (\sigma^2/n)\chi^2_{n-k}$, and $P_{W, \sigma^2} = \int P_{\beta,\sigma^2}dW(\beta)$; the Bayes marginal of $P_{\beta, \sigma^2}$ relative to the prior $W$ on $\beta$. The distribution of $\hat\sigma^2$ only depends on $\sigma^2$, so the likelihood ratio
\begin{align*}
   E^* =  \frac{q_{\sigma_1^2}(\hat\sigma^2)}{q_{\sigma_0^2}(\hat\sigma^2)} = \left(\frac{\sigma_0}{\sigma_1}\right)^{n-k}\exp\left(-\frac{n\hat\sigma^2}{2}\left(\frac{1}{\sigma_1^2}-\frac{1}{\sigma_0^2}\right)\right)
\end{align*}
is an e-variable for $H_0'$, and therefore also for $H_0$. This e-variable does not depend on $\hat\beta$, so it is not immediately clear why $E^*$ should be the GRO(W) e-variable.

Following \citet[Appendix~A.1 and A.2]{grunwald2024safe},
\begin{align}
\begin{split} \label{eq:lin-reg-GRO-upper-bound}
    \sup_{E \in \calE'}\E_{\beta, \sigma_1^2}[\log(E)] &\leq \inf_{W_0} \KL(P_{\beta, \sigma_1^2} \otimes Q_{\sigma_1^2} \| P_{W_0, \sigma_0^2} \otimes Q_{\sigma_0^2})\\ &= \KL(Q_{\sigma_1^2}\| Q_{\sigma_0^2}) + \inf_{W_0} \KL(P_{\beta, \sigma_1^2} \| P_{W_0, \sigma_0^2} ),
\end{split} \\
\begin{split} \label{eq:lin-reg-GROW-upper-bound}
    \sup_{E \in \calE'}\inf_{\beta \in \Re^k}\E_{\beta, \sigma_1^2}[\log(E)] &\leq \inf_{W_0,W_1} \KL(P_{W_1, \sigma_1^2} \otimes Q_{\sigma_1^2} \| P_{W_0, \sigma_0^2} \otimes Q_{\sigma_0^2})\\ &= \KL(Q_{\sigma_1^2}\| Q_{\sigma_0^2}) + \inf_{W_0,W_1} \KL(P_{W_1, \sigma_1^2} \| P_{W_0, \sigma_0^2} ),
\end{split}
\end{align}
where $W_0$ and $W_1$ run over the set of priors on $\beta \in \Re^k$. That is, \eqref{eq:lin-reg-GRO-upper-bound} presents an upper bound on the maximal log-utility that $E \in \calE'$ can achieve under the point alternative $(\beta, \sigma_1^2) \in \Theta_1$, and \eqref{eq:lin-reg-GROW-upper-bound} presents an upper bound on the maximal worst case log-utility that $E \in \calE'$ can achieve under the composite alternative $\Theta_1$. Clearly, $\E_{\beta, \sigma_1^2}[\log(E^*)] = \KL(Q_{\sigma_1^2}\| Q_{\sigma_0^2})$, so $E^*$ is the GRO, respectively GROW, e-variable if it can be shown that the infimum in \eqref{eq:lin-reg-GRO-upper-bound}, respectively \eqref{eq:lin-reg-GROW-upper-bound}, is $0$.

Let us first assume that $\sigma_1^2 \geq \sigma_0^2$. One, then, easily shows that $P_{\beta, \sigma_1^2} = P_{W_0, \sigma_0^2}$ with $W_0 = N(\beta, (\sigma_1^2 - \sigma_0^2)(X^\top X)^{-1})$, and thus $\KL(P_{\beta, \sigma_1^2} \| P_{W_0, \sigma_0^2}) = 0$. Whence, $E^*$ is indeed GRO for the point alternative $(\beta, \sigma_1^2) \in \Theta_1$. Because this derivation holds for every $\beta \in \Re^k$, and because $E^*$ does not depend on $\beta$, we may conclude that $E^*$ is GRO for every point alternative $(\beta, \sigma_1^2) \in \Theta_1$, and we thus say, following \citet{deRoos2025}, that $E^*$ is \textit{uniformly} GRO for $\Theta_1$. Clearly, this also implies that $E^*$ is GROW for $\Theta_1$.

The above argument does not work if $\sigma_0^2 > \sigma_1^2$, because $(\sigma_1^2 - \sigma_0^2)(X^\top X)^{-1}$ is, then, negative definite. The argument that $E^*$ is GROW is only a slight adaptation, though. Namely, take $W_0 = N(0, \sigma_1^2(X^\top X)^{-1})$ and $W_1 = N(0, \sigma_0^2(X^\top X)^{-1})$, so that $P_{W_1, \sigma_1^2}  = P_{W_0, \sigma_0^2}$. Then, $\KL(P_{W_1, \sigma_1^2} \| P_{W_0, \sigma_0^2}) = 0$, so $E^*$ is indeed GROW for $\Theta_1$. When $\sigma_0^2 > \sigma_1^2$, $E^*$ is not GRO for any point alternative $(\beta, \sigma_1^2) \in \Theta_1$, which can be seen to be a straightforward consequence of Theorem~1 by \citet{grunwald2024optimal}.

\subsection{Permutation-based e-variables} \label{sec:permutation}

Let $X_1, \ldots, X_n \in \Re$ be a sequence of data, and suppose it is exchangeable under every element of the null and alternative. This is, for instance, satisfied when $X_1, \ldots, X_n$ is an i.i.d. sequence.

The ordered sample $S = (X_{(1)}, \ldots, X_{(n)})$ is sufficient for the union of the null and alternative, and $\E_\Theta[f(X_1, \ldots, X_n) \mid S] = (n!)^{-1}\sum_{\pi}f(X_{\pi(1)}, \ldots, X_{\pi(n)})$, where $\pi$ runs over the set of permutations on $\{1, \ldots, n\}$. According to Theorems~\ref{thm:rb-evals} and \ref{thm:rb-log-util-general}, the e-variable $E = E(X_1, \ldots, X_n)$ can be improved in terms of expected concave utility by Rao-Blackwellizing it into $G = (n!)^{-1}\sum_\pi E(X_{\pi(1)}, \ldots, X_{\pi(n)})$.

The fact that averaging improves expected utility was first discussed by \citet{wasserman2020universal} and then proved for the log-utility in \citet[Proposition 5.3]{ramdas2025hypothesis}. We show that permutation-averaging of an e-variable for exchangeable data is not merely an ad hoc symmetrization trick: it is exactly Rao-Blackwellization with respect to the ordered sample, thus explaining the increase of expected utility.

Whenever $E$ is invariant under permutations, Rao-Blackwellization has no effect and $E$ and $G$ will coincide. On the other hand, in a practical setting where no canonical e-variable is readily available, $E$ may, for instance, be constructed via sample splitting \citep[see, e.g.,][]{wasserman2020universal}. In such a case, Rao-Blackwellization will result in a strict improvement whenever the split-dependent e-variable is not conditionally constant given the ordered sample, but the procedure will be computationally costly at large sample sizes as it involves recalculating the e-variable over many permutations. Furthermore, because the GRO e-variable is essentially unique, our result also demonstrates that the GRO e-variable, at least up to a $P_{\theta_1}$-negligible set, is permutation invariant. Thus, if an e-variable is not permutation invariant, again up to a $P_{\theta_1}$-negligible set, then it cannot possibly be that it is GRO. Hence, when assessing whether an e-variable is GRO for exchangeable data, it is almost always wise to verify that it is permutation invariant, if only for the sake of a sanity check.

\section{Discussion}
Results stating that, in certain situations, the search for optimal e-variables can be restricted to a smaller subset of e-variables, are common in the e-variable literature. These results often pertain to the concept of \textit{admissibility} \citep{chugg2026admissibility, larsson2026constraints, clerico2024optimal}. An interesting research direction would be to define and study a notion of admissibility that is consistent with Theorems~\ref{thm:rb-evals} and~\ref{thm:rb-log-util-general}. Such a notion would, at minimum, partially characterize admissible e-variables as those that cannot be improved through Rao-Blackwellization in terms of expected concave utility. While this partial characterization was already adequate for the examples we discussed, further development is needed for settings where restricting to e-variables that are functions of the sufficient statistic does not yield a dramatic simplification of the underlying optimization problem. Future research should therefore pursue two directions: (i) defining and studying a notion of admissibility in relation to Theorems~\ref{thm:rb-evals} and~\ref{thm:rb-log-util-general}, and (ii) characterizing the admissible classes of e-variables in these more complex settings.

\section*{Acknowledgement}
We thank Nick Koning and three anonymous referees for their valuable comments, which greatly improved this work. DdR and PG have been supported by ERC grant 101142168, FLEX, awarded to PG. BC and AR acknowledge NSF grants IIS-2229881 and DMS-2310718.

{\small
\DeclareRobustCommand{\sortandprefix}[3]{#3} 
\bibliographystyle{abbrvnat}
\bibliography{main}
}

\appendix

\section*{Appendix}
The appendix contains additional results regarding the Rao-Blackwellization of other types of e-statistics (Appendix~\ref{sec:extensions}), two additional examples demonstrating the use of our results (Appendix~\ref{sec:additional-examples}), a slight generalization of Jensen's inequality for $[0, \infty]$-valued random variables we needed for the proof of Theorem~\ref{thm:rb-evals} (Appendix~\ref{sec:jensen}), and the proofs of Theorem~\ref{thm:rb-log-util-general} and Lemma~\ref{lem:log-utility-defined} (Appendix~\ref{sec:omitted-proofs}).

\section{Extensions to other types of e-statistics}
\label{sec:extensions}
\subsection{Introduction}
In this section, we will demonstrate that Theorems~\ref{thm:rb-evals} and~\ref{thm:rb-log-util-general} extend to three other types of e-statistics: compound e-variables, asymptotic e-variables, and e-processes. Asymptotic and compound e-variables, both introduced by \citet{ignatiadis2024asymptotic}, are important tools in multiple testing; most notably, any procedure
which controls the false discovery rate can be recovered (or improved)
using the e-Benjamini-Hochberg procedure \citep{wang2022false} with some set of compound e-variables~\citep[Theorem 4.4]{ignatiadis2024asymptotic}. E-processes are the sequential analogues of e-variables and are fundamental tools in anytime-valid sequential inference \citep{ramdas2023game}.

\subsection{Compound e-variables}
We first treat compound e-variables. To that end, consider $K$ (not necessarily disjoint) null hypotheses $\Theta_0^1,\dots,\Theta_0^K$ and set $\Theta_0 = \cup_j \Theta_0^j$. A \emph{compound} e-variable for $\Theta_0^1,\dots,\Theta_0^K$ is a tuple of nonnegative random variables $(E_1,\dots,E_K)$ such that
\begin{equation}
\label{eq:compound-eval}
  \sum_{k:\theta\in\Theta_0^k} \E_\theta[E_k]\leq K,\quad \textnormal{for all } \theta\in\Theta_0.
\end{equation}
For $j=1,\dots,K$, let $S_j$ be a sufficient statistic for $\Theta_j\equiv \Theta_0^j \cup \Theta_1$ and define $G_j = \E_{\Theta_j} [E_j \mid S_j]$. We again use the notation $\E_{\Theta_j}[\cdot \mid S_j]$ since the quantity is independent of $\theta$ by virtue of $S_j$ being sufficient.
Since $\E_\theta[G_j] = \E_\theta[E_j]$ for any $\theta\in \Theta_j$,~\eqref{eq:compound-eval} holds for $G_1,\dots,G_K$, meaning the collection is a compound e-variable. Further, $\E_\theta[f(G_j)] \geq \E_\theta[f(E_j)]$ for all $\theta\in \Theta_1 \cup \Theta_0^j$ and concave utility functions $f$ by Theorem~\ref{thm:rb-evals} and we can likewise apply Theorem~\ref{thm:rb-log-util-general}. We summarize as follows.

\begin{corollary} \label{cor:compound}
    Let $(E_1,\dots,E_K)$ be a compound e-variable for the null hypotheses $\Theta_0^1,\dots,\Theta_0^K$. For $j=1,\dots,K$, let $S_j$ be a sufficient statistic for $\Theta_0^j\cup \Theta_1$ and define $G_j = \E_{\Theta_j}[E_j \mid  S_j]$. Then $\E_\theta[f(G_j)] \geq \E_\theta[f(E_j)]$ for any concave $f$ and $\theta\in \Theta_1 \cup \Theta_0^j$, and
    \begin{equation}
      \sum_{k:\theta\in\Theta_0^k\cup \Theta_1} \E_\theta[f(G_k)]\geq \sum_{k:\theta\in\Theta_0^k\cup \Theta_1} \E_\theta[f(E_k)],\textnormal{~ for all } \theta\in \Theta_0 \cup \Theta_1,
    \end{equation}
    assuming all expectations are well-defined. Moreover, $\E_\theta[\log(E_j/G_j)]\leq 0$ for all $j$ and $\theta\in\Theta_0^j\cup\Theta_1$.
\end{corollary}

While this result is reminiscent of the ``conditional calibration'' approach of \citet{lee2024boosting} for improving e-Benjamini-Hochberg, the target of the conditioning is different. In our case, we Rao-Blackwellize the (compound) e-variables themselves with respect to a sufficient statistic. By contrast, \citet{lee2024boosting} do not condition the e-variables directly. Instead, for each null they condition the per-hypothesis false discovery rate contribution on a (null-)sufficient statistic and then define boosted e-variables indirectly via this conditionally calibrated contribution.

\subsection{Asymptotic e-variables}
Let us now briefly turn to asymptotic e-variables. An asymptotic e-variable is a sequence of nonnegative random variables $(E_t)_{t \geq 0}$ such that $\sup_{\theta\in \Theta_0}\limsup_{t\to\infty} \E_{\theta} [E_t] \leq 1$. A \emph{uniform} asymptotic e-variable is a sequence $(E_t)_{t \geq 0}$ such that $\limsup_{t\to\infty} \sup_{\theta\in \Theta_0}\E_{\theta} [E_t] \leq 1$. The latter condition is much stronger than the former as the limit must hold uniformly over $\Theta_0$.

We assume that $(E_t)_{t \geq 0}$ is adapted to some filtration $(\calF_t)_{t \geq 0}$, and we denote by $(S_t)_{t \geq 0}$ a sequence of sufficient statistics on the filtered sample space $(\Omega, (\calF_t)_{t \geq 0})$. That is, for each $t$, $S_t$ is $\calF_t$-measurable and is sufficient for $\Theta$ in the \textit{$\mathcal{F}_t$-experiment} induced by coarsening the sample space to $(\Omega, \mathcal{F}_t)$, meaning that for every $\mathcal{F}_t$-measurable $h$, all $\E_\theta(h(X) \mid S_t)$, provided that they are defined, have a universal version, which we denote by $\E_\Theta(h(X) \mid S_t)$.

Applying Theorems~\ref{thm:rb-evals} and~\ref{thm:rb-log-util-general} to each $E_t$ (and recalling that the concave
improvement does not require $E_t$ to be an e-variable; see Remark~\ref{rem:any-variable}) yields the following result.

\begin{corollary} \label{cor:sup-asymptotic}
    Let $(S_t)_{t \geq 0}$ be a sequence of sufficient statistics for $\Theta$ on $(\Omega,(\calF_t)_{t \geq 0})$. Let $(E_t)_{t \geq 0}$ be a (uniform) asymptotic e-variable.
    For each $t$, set $G_t = \E_\Theta[E_t \mid S_t]$. Then (i) $(G_t)_{t \geq 0}$ is a (uniform) asymptotic e-variable and (ii) $\E_\theta[f(G_t)] \geq \E_\theta[f(E_t)]$ for all concave $f$, assuming both expectations are defined. Moreover, $\E_\theta[\log(E_t/G_t)]\leq 0$ for all $t$ and $\theta\in\Theta$.
\end{corollary}

\subsection{E-processes}
To move to the sequential setting and extend these results to e-processes, we, again, assume that the measurable space $(\Omega,\calF)$ admits a filtration $(\calF_t)_{t\geq 0}$. An \textit{e-process} for $\Theta_0$ with respect to $(\calF_t)_{t \geq 0}$ is a sequence $(E_t)_{t \geq 0}$ of nonnegative random variables such that $E_t$ is $\calF_t$-measurable for all $t\geq 0$ and $E_\tau$ is an e-variable for $\Theta_0$ at every $(\calF_t)_{t \geq 0}$-stopping time $\tau$. Recall that $\tau$ is a $(\calF_t)_{t \geq 0}$-stopping time if the event $\{\tau = t\}$ is $\calF_t$-measurable for all $t \geq 0$.

Similar to the conclusions of Corollary~\ref{cor:sup-asymptotic}, conditioning $E_t$ on a sufficient statistic $S_t$ will yield an improvement in terms of expected concave utility, but the corresponding process need not be an e-process anymore.

As a remedy, we introduce the notion of a \textit{sufficient filtration for $\Theta$} and Rao-Blackwellize our e-processes based on these. We say that the filtration $(\mathcal{S}_t)_{t \geq 0} \subset (\mathcal{F}_t)_{t \geq 0}$ is sufficient for $\Theta$ if for every $t \geq 0$, $\mathcal{S}_t$ is sufficient for $\Theta$ in the $\mathcal{F}_t$-experiment. Here, the sufficiency of $\mathcal{S}_t$ is defined analogously to the definition of a sufficient statistic in the $\mathcal{F}_t$-experiment. Moreover, we denote for $\mathcal{F}_t$-measurable $h$ the universal version of all $\E_\theta(h(X) \mid \mathcal{S}_t)$ by $\E_\Theta(h(X) \mid \mathcal{S}_t)$, provided that all these conditional expectations exist.
\begin{corollary} \label{cor:sup-rb-e-proc}
Let $(\calS_t)_{t \geq 0}$ be a sufficient filtration for $\Theta$ satisfying $\calS_t \subseteq \calF_t$ for all $t\geq 0$.
    Let $(E_t)_{t \geq 0}$ be an e-process for $\Theta_0$ with respect to $(\calF_t)_{t \geq 0}$ and define $G_t = \E_{\Theta}[E_t \mid  \calS_t]$. Then (i) $(G_t)_{t \geq 0}$ is an e-process for $\Theta_0$ with respect to $(\calS_t)_{t \geq 0}$ and (ii) for all $\theta\in\Theta$,  $\E_\theta[f(G_\tau)] \geq \E_\theta[f(E_\tau)]$ for all $(\calS_t)_{t \geq 0}$-stopping times $\tau$ and all concave utility functions $f$, provided both sides of the inequality are defined. Moreover, $\E_\theta[\log(E_\tau/G_\tau)]\leq 0$ for any $(\calS_t)_{t \geq 0}$-stopping time $\tau$.
\end{corollary}
\begin{proof}
    First we observe that $G_\tau = \E_\theta[E_\tau \mid \calS_\tau]$ $P_\theta$-a.s for any $(\calS_t)_{t \geq 0}$-stopping time $\tau$. Indeed, consider any $H\in\calS_\tau$. Put $H_t = H \cap\{\tau = t\}$, which is $\calS_t$-measurable. Then
\begin{equation*}
    \int_H G_\tau d P_\theta = \sum_{t\geq 0} \int_{H_t} G_t d P_\theta = \sum_{t\geq 0} \int_{H_t} E_t d P_\theta = \int_H E_\tau d P_\theta.
\end{equation*}
This implies that $(G_t)_{t \geq 0}$ is an e-process for $\Theta_0$: For $\theta_0 \in \Theta_0$, $\E_{\theta_0} [G_\tau ] = \E_{\theta_0}[\E_{\theta_0}[E_\tau\mid \calS_\tau]] = \E_{\theta_0}[E_\tau] \leq 1$, where the final inequality follows since the event $\{\tau=t\}$, by virtue of being $\calS_t$-measurable, is also $\calF_t$-measurable.
Next, following the same reasoning as in Theorem~\ref{thm:rb-evals}, for any $\theta\in\Theta$ and any concave utility function $f$, we have
\begin{align*}
    \E_\theta[f(E_\tau)] = \E_\theta[\E_\theta[f(E_\tau)\mid \mathcal{S}_\tau]] \leq \E_\theta[f(\E_\theta[E_\tau\mid \mathcal{S}_\tau])] = \E_\theta[f(G_\tau)],
\end{align*}
provided that both $\E_\theta[f(E_\tau)]$ and $\E_\theta[f(G_\tau)]$ exist. The final statement follows from applying Theorem~\ref{thm:rb-log-util-general} to $G_\tau$ and $E_\tau$, completing the argument.
\end{proof}
To make Corollary~\ref{cor:sup-rb-e-proc} more concrete, consider an adaptive experiment where data $X_1,\dots,X_t$ are revealed one observation at a time, $\calF_t = \sigma(X_1,\dots,X_t)$, and $S_t$ is a sufficient statistic for $\Theta$ at time $t$. One naturally wishes to apply Corollary~\ref{cor:sup-rb-e-proc} with $\calS_t =\sigma(S_1,\dots,S_t)$.

Counterintuitively, however, defining $\calS_t$ in this way may not result in a sufficient filtration. In fact,
\citet{burkholder1961sufficiency} gives an example of a $\sigma$-field that is a superset of a sufficient $\sigma$-field, but is itself not sufficient. Burkholder's example is rather pathological, however, and we can make any such examples vanish by assuming that our model is locally dominated; that is, the restrictions $\{P_\theta|_{\mathcal{F}_t}\}_{\theta \in \Theta}$ of $P_\theta$ to $\mathcal{F}_t$ are dominated by some $\sigma$-finite measure $\nu_t$ on $\mathcal{F}_t$. In this case, the Neyman factorization criterion of sufficiency applies \citep{bahadur1954sufficiency, halmos1949application}:
\begin{align*}
    \frac{dP_\theta|_{\mathcal{F}_t}}{d\nu_t}(x) = h_t(x)g_t(\theta ;S_t(x)) = h_t(x)g_t^*(\theta;S_1(x), \ldots, S_t(x)),
\end{align*}
where $g_t^*(\theta; s_1, \ldots, s_t) = g_t(\theta; s_t)$. Thus, $(S_1, \ldots, S_t)$ is also sufficient for $\Theta$ in the $\mathcal{F}_t$-experiment. To summarize: If the family $\{P_\theta\}_{\theta\in\Theta}$ is locally dominated, then $\calS_t = \sigma(S_1,\dots,S_t)$ defines a sufficient filtration, and the conclusion of Corollary~\ref{cor:sup-rb-e-proc} applies to it.

On the other hand, for Corollary~\ref{cor:sup-rb-e-proc} to be non-trivial, $\calS_t$ should be a coarser filtration than $\calF_t$ (i.e., a strict subset).
Consider a sequentialized version of Example~\ref{ex:bern-naive-eval}, where $X_t \sim \mathrm{Ber}(\theta)$ and $S_t = \sum_{i=1}^t X_i$ is sufficient at time $t$. Then clearly $\calS_t = \calF_t$ so $G_t = E_t$ and Rao-Blackwellization does not buy us anything.
There are many natural situations in which $\calS_t\subsetneq\calF_t$, however. The maximum likelihood estimator in linear regression is one such example; see Appendix~\ref{sec:linear-regression-fixed-variance} below.
One could also imagine starting the process after a fixed-time, say $M$, so that $\calS_t = \sigma(S_M,\dots,S_t)$. This might correspond to a researcher pledging to not look at the data until after some initial ``burn-in period.''

\section{Additional examples}
\label{sec:additional-examples}
\subsection{Introduction}
In this section we discuss two examples similar to, and intended to complement, the one in Section~\ref{sec:linear-reg}. First, in Appendix~\ref{sec:linear-regression-fixed-variance} below, we derive the GRO e-variable for testing a constrained subset of coefficients in a linear model with fixed variance, and second, in Appendix~\ref{sec:pareto-grow} below we derive the GROW e-variable for testing the shape of a Pareto distribution. Both examples partially appeared in \citet{deRoos2025}, but we emphasize that the present presentation differs substantially. Most notably, the linear regression example is worked out in far greater detail here, leading us to conclude that the e-variable we derive is uniformly GRO, while the Pareto example presented here fully relies on Rao-Blackwellization, whereas \citet{deRoos2025} required different concepts.

\subsection{GRO e-variable for testing a constrained subset of coefficients in a linear model} \label{sec:linear-regression-fixed-variance}
Consider a linear regression model $Y \sim N(X\beta, \sigma^2 I_n)$ with $\beta = (\beta_0, \beta_1) \in \Re^d \times \Re^{k-d}$, $\sigma^2 \in (0, \infty)$ known, and $X \in \Re^{n \times k}$ known and of full rank. We are interested in testing the null $\Theta_0 = \{(0, \beta_1) \mid \beta_1 \in \Re^{k-d}\}$ against the point alternative $ \Theta_1 = \{(\beta_0^*, \beta_1^*)\}$ for some $\beta_0^* \neq 0$ and $\beta_1^* \in \Re^{k-d}$. Formally, we are testing
\begin{equation*}
\begin{split}
        &H_0: Y \sim N\left(X \begin{pmatrix} 0 \\ \beta_1 \end{pmatrix}, \sigma^2I_n\right), \quad \beta_1 \in \Re^{k-d}, \\
        &H_1:Y \sim N\left(X \begin{pmatrix} \beta_0^* \\ \beta_1^* \end{pmatrix}, \sigma^2I_n\right).
\end{split}
\end{equation*}
The maximum likelihood estimator $\hat\beta = (X^\top X)^{-1} X^\top Y$
is sufficient for $\beta$ when $\sigma^2$ is fixed, and $\hat\beta \sim N(\beta, \sigma^2 (X^\top X)^{-1})$. Hence, by Theorem~\ref{thm:rb-evals}, the GRO e-variable for testing
\begin{equation*}
\begin{split}
        &H_0': \hat\beta \sim N\left( \begin{pmatrix} 0 \\ \beta_1\end{pmatrix}, \sigma^2(X^\top X)^{-1}\right), \, \beta_1 \in \Re^{k-d}, \\
        &H_1': \hat\beta \sim N\left( \begin{pmatrix} \beta_0^* \\ \beta_1^*\end{pmatrix}, \sigma^2(X^\top X)^{-1}\right),
\end{split}
\end{equation*}
is the GRO e-variable for testing $H_0$ versus $H_1$. The latter testing problem corresponds to testing a Gaussian location family with a constrained mean against some fixed other mean, and has been studied by \citet{grunwald2024optimal, spector2023discussion}. In fact, denoting the density of a multivariate Gaussian distribution with mean $m$ and covariance $\Sigma$ by $\rho(\cdot ; m, \Sigma)$, it turns out that the GRO e-variable for this testing problem is the simple likelihood ratio
\begin{equation} \label{eq:GRO-e-var-lin-reg}
    \frac{\rho( \hat \beta; \beta^*, \sigma^2(X^\top X)^{-1} )}{\rho( \hat \beta; \beta', \sigma^2(X^\top X)^{-1} )} = \exp\left(\sigma^{-2}\left(\hat{\beta}_0^\top \Psi \beta_0^* - \frac{1}{2}{\beta_0^*}^\top \Psi \beta_0^*\right)\right)
\end{equation}
with $\beta^* = (\beta_0^*, \beta_1^*)$, $ \beta' = (0, \Sigma_{11}^{-1}\Sigma_{01}^\top \beta_0^* + \beta_1^*)$, and $\Psi = \Sigma_{00} - \Sigma_{01}\Sigma^{-1}_{11} \Sigma_{01}^\top$. Here, $\Sigma_{00}$, $\Sigma_{11}$, and $\Sigma_{01}$ are obtained by decomposing $X^\top X$ as follows:
\begin{equation*}
    X^\top X = \begin{pmatrix}
        \Sigma_{00} & \Sigma_{01} \\
        \Sigma_{01}^\top & \Sigma_{11}
    \end{pmatrix}, \, \Sigma_{00} \in \Re^{d \times d}, \, \Sigma_{11} \in \Re^{(k-d) \times (k-d)}, \, \Sigma_{01} \in \Re^{d \times (k-d)}.
\end{equation*}
The right-hand side of \eqref{eq:GRO-e-var-lin-reg} does not depend on the specific choice of $\beta_1^*$, so this e-variable is in particular uniformly GRO for the composite alternative $\Theta_1 = \{(\beta_0^*, \beta_1) \mid \beta_1 \in \Re^{k-d}\}$.

\subsection{GROW e-variable for testing the shape of a Pareto distribution} \label{sec:pareto-grow}
Let $X_1, \ldots, X_n$ be i.i.d. $\Par(m, \alpha)$, $m \in (0,\infty)$, $\alpha > 0$. That is, $(X_1, \ldots, X_n)$ has the joint Lebesgue density $p_{m,\alpha}(x_1, \ldots, x_n)$, given by
\[
\prod_{i=1}^n \frac{\alpha m^\alpha}{x_i^{\alpha+1}}\ind\{x_i \geq m\} = \alpha^nm^{n\alpha}\exp\left(-(\alpha+1)\sum_{i=1}^n \log\frac{x_i}{x_{(1)}}\right)\ind\{x_{(1)} \geq m\},
\]
where $x_{(1)} = \min(x_1, \ldots, x_n)$. By the Neyman factorization criterion for sufficiency, the statistic $(X_{(1)}, \sum_{i=1}^n\log(X_i/X_{(1)}))$ is sufficient for $(m, \alpha)$. Let the null $\mathcal{P}_{\alpha_0}$ be given by $H_0:\alpha = \alpha_0$, and the alternative $\mathcal{P}_{\alpha_1}$ by $H_1:\alpha = \alpha_1$. Formally, we consider the testing problem
\begin{align*}
    &H_0: X_1, \ldots, X_n \overset{\textnormal{i.i.d.}}{\sim} \Par(m, \alpha_0), \, m \in (0, \infty), \\
    &H_1: X_1, \ldots, X_n \overset{\textnormal{i.i.d.}}{\sim} \Par(m, \alpha_1), \, m \in (0, \infty).
\end{align*}

A straightforward computation shows that for any $P_{m, \alpha_1} \in \mathcal{P}_{\alpha_1}$, $\KL(P_{m, \alpha_1} \| P_{m, \alpha_0}) < \infty$, so by Lemma~\ref{lem:log-utility-defined} we find that the GROW criterion for this testing problem is well-defined.

\citet{malik1970estimation} has shown that
\[
    \left(X_{(1)}, \sum_{i=1}^n\log\frac{X_i}{X_{(1)}}\right) \sim \Par(m, n\alpha) \otimes \Gamma(n-1, \alpha),
\]
where $\Gamma(n-1, \alpha)$ denotes the Gamma distribution with shape parameter $n-1$ and rate parameter $\alpha$. Thus, by Theorem~\ref{thm:rb-evals}, the GROW e-variable for testing
\begin{align*}
    H_0':\left(X_{(1)}, \sum_{i=1}^n\log \frac{X_i}{X_{(1)}}\right) \sim \Par(m, n\alpha_0) \otimes \Gamma(n-1, \alpha_0), \ m\in(0,\infty), \\
    \quad H_1':\left(X_{(1)}, \sum_{i=1}^n\log \frac{X_i}{X_{(1)}}\right) \sim \Par(m, n\alpha_1) \otimes \Gamma(n-1, \alpha_1), \ m \in (0,\infty),
\end{align*}
will be the GROW e-variable for testing $H_0$ versus $H_1$.

Let us denote $P_{m, \alpha} = \Par(m, \alpha)$, $Q_{\alpha} = \Gamma(n-1, \alpha)$, and $P_{W, \alpha} = \int P_{m, \alpha} dW(m)$; the Bayes marginal of $P_{m, \alpha}$ with respect to the prior $W$ on $m$. Because the distribution of $\sum_{i=1}^n\log({X_i}/{X_{(1)}})$ only depends on $\alpha$, the likelihood ratio
\[
E^* =\frac{q_{\alpha_1}\left(\sum_{i=1}^n\log\frac{X_i}{X_{(1)}}\right)}{q_{\alpha_0}\left(\sum_{i=1}^n\log\frac{X_i}{X_{(1)}}\right)} = \left(\frac{\alpha_1}{\alpha_0}\right)^{n-1}\exp\left(-(\alpha_1-\alpha_0)\sum_{i=1}^n\log\frac{X_i}{X_{(1)}}\right),
\]
is an e-variable for $H_0'$ and therefore also for $H_0$.

Denoting the set of e-variables for $H_0'$ by $\mathcal{E}'$, we find, following \citet[Appendix A.2]{grunwald2024safe}, that
\begin{equation} \label{eq:KL-upper-bound-Pareto}
\begin{split}
    \sup_{E \in \calE'}\inf_{m \in (0,\infty)}\E_{m, n\alpha_1}[\log(E)] &\leq \inf_{W_0,W_1} \KL(P_{W_1, n\alpha_1} \otimes Q_{ \alpha_1} \| P_{W_0, n\alpha_0} \otimes Q_{\alpha_0})\\ &= \KL(Q_{\alpha_1}\| Q_{\alpha_0}) + \inf_{W_0,W_1} \KL(P_{W_1, n\alpha_1} \| P_{W_0, n\alpha_0} ),
\end{split}
\end{equation}
where $W_0$ and $W_1$ run over the set of priors on $m$. By definition, $\E_{m, \alpha_1}[\log(E^*)] = \KL(Q_{\alpha_1}\| Q_{\alpha_0})$, so we may conclude that $E^*$ is GROW if we can demonstrate that the infimum in the right-hand side of \eqref{eq:KL-upper-bound-Pareto} is $0$.

The proof is not as straightforward as in Section~\ref{sec:linear-reg}, and we require the additional assumption that $n\alpha_1 > 1$, which is obviously satisfied for a large sample size $n$.

Now, let $W_u$ denote the $\Gamma(u,u)$-distribution with $u \in (0,\infty)$. It suffices to show that $\KL(P_{W_u, n\alpha_1} \| P_{W_u, n\alpha_0})$ tends to zero as $u \downarrow 0$. The density $f_{W_u, \alpha}$ of $X\sim P_{W_u, \alpha}$ is
\begin{align*}
        f_{W_u, \alpha}(x) &= \int_0^x \frac{\alpha m^\alpha}{x^{\alpha+1}}\frac{u^u m^{u-1}}{\Gamma(u)}e^{-um}\mathrm{d}m \\
        &= \frac{\alpha u^u}{x^{\alpha+1}\Gamma(u)}\int_0^x m^{\alpha+u-1}e^{-um}\mathrm{d}m \\
        &= \frac{x^{\alpha+u}\alpha u^u}{x^{\alpha+1}\Gamma(u)}\int_0^1 y^{\alpha+u-1}e^{-uyx}\mathrm{d}y.
\end{align*}
Now, since
\begin{align*}
    \KL(P_{W_u, n\alpha_1} \| P_{W_u, n\alpha_0}) = \E_{W_u, n\alpha_1}\left(\log \frac{f_{W_u, n\alpha_1}(X)}{f_{W_u, n\alpha_0}(X)}\right),
\end{align*}
we can bound $\KL(P_{W_u, n\alpha_1} \| P_{W_u, n\alpha_0})$ by uniformly bounding the logarithm:
\begin{align*}
    \log\frac{f_{W_u, n\alpha_1}(x)}{f_{W_u, n\alpha_0}(x)} = \log \left(\frac{\alpha_1 \int_0^1 y^{n\alpha_1+u-1}e^{-uyx}\mathrm{d}y }{\alpha_0 \int_0^1 y^{n\alpha_0+u-1}e^{-uyx}\mathrm{d}y}\right)\leq \log\left(\frac{n\alpha_1(n\alpha_0 + u)}{n\alpha_0(n\alpha_1+u)}\right) + ux,
\end{align*}
because
\begin{align*}
    \frac{e^{-ux}}{\alpha+u} \leq \int_0^1 y^{\alpha+u-1}e^{-uyx}\mathrm{d}y \leq \frac{1}{\alpha+u}.
\end{align*}
Therefore,
\begin{equation} \label{eq:KL-upper-bound-proof-pareto}
    \KL(P_{W_u, n\alpha_1} \| P_{W_u, n\alpha_0}) \leq \log\left(\frac{n\alpha_1(n\alpha_0 + u)}{n\alpha_0(n\alpha_1+u)}\right) + u\E_{W_u, n\alpha_1}(X).
\end{equation}
The first summand tends to $0$ as $u \downarrow 0$, and
\begin{align*}
     \E_{W_u, n\alpha_1}(X) = \E_{m \sim W_u}\left(\E_{m, n\alpha_1}(X)\right) = \E_{m \sim W_u}\left(\frac{n\alpha_1m}{n\alpha_1-1}\right) = \frac{n\alpha_1}{n\alpha_1-1},
\end{align*}
which exists because $n\alpha_1 > 1$ by assumption. Indeed, the second summand of \eqref{eq:KL-upper-bound-proof-pareto} also tends to $0$ as $u \downarrow 0$, so
\begin{align*}
    \lim_{u \downarrow 0} \KL(P_{W_u, n\alpha_1} \| P_{W_u, n\alpha_0}) = 0,
\end{align*}
which is what we wanted to show.

\section{A general version of Jensen's inequality}
\label{sec:jensen}
In this section, we present and prove a general version of Jensen's inequality for nonnegative (and potentially nonintegrable) random variables $X: \Omega \to [0, \infty]$. A similar result for nonnegative concave utility functions has already been established in~\citet[Proposition A.7]{wang2025extended}. However, since we are working with general concave utility functions in this paper, we must generalize this result to its fullest extent.

We consider the concave function $f$ to be defined as a function from $(0, \infty)$ to $\mathbb{R}$, and we extend the domain of $f$ to $[0, \infty]$ by defining $f(0) = \lim_{x\downarrow 0}f(x)$ and $f(\infty) = \lim_{x\to \infty}f(x)$, which exist as elements of $[-\infty, \infty]$. Because $f$ is continuous on $(0, \infty)$, this definition ensures that $f$ is also continuous as a function from $[0, \infty]$ to $[-\infty, \infty]$, provided that we endow $[0, \infty]$ and $[-\infty, \infty]$ with their respective order topologies. (The latter ensures that we can always take limits inside $f$ during the proof of the following lemma, which we will do without further comment.)

Recall that the expectation of a random variable $Y:\Omega \to [-\infty, \infty]$ is defined if $\min(\E[Y^+], \E[Y^-]) < \infty$, where $x^+ = \max(0, x)$, and $x^- = \max(0, -x)$ respectively denote the positive and negative parts of $x$. This condition ensures that $\E[Y] = \E[Y^+]-\E[Y^-]$ is well-defined as an element of $[-\infty, \infty]$. Furthermore, conditional expectations can, then, also be defined as $\E[Y \mid  \mathcal{G}] = \E[Y^+ \mid  \mathcal{G}] - \E[Y^- \mid  \mathcal{G}]$, which almost surely takes its values in $[-\infty, \infty]$. For a full treatment on general (conditional) expectation, we refer to \citet{shiryaev2016probability}.
\begin{lemma} \label{lem:general-jensen-ineq}
    Let $X$ be a nonnegative random variable, and $f$ a concave function as described above. If $\E[f(X)]$ is defined, then for any $\sigma$-field $\mathcal{G} \subset \mathcal{F}$,
    \[
    \E[f(X) \mid  \mathcal{G}] \leq f(\E[X \mid  \mathcal{G}]) \quad \textnormal{a.s.}
    \]
    Additionally, if $\E[f(\E[X \mid  \mathcal{G}])]$ is also defined, then $\E[f(X)] \leq \E[f(\E[X \mid  \mathcal{G}])]$.
\end{lemma}
\begin{proof}
    We first assume that $\E[f(X)^+] < \infty$. Define, for $0 < \varepsilon < \delta$
    \[
    X^\varepsilon = X \vee \varepsilon, \quad X_\delta = X\wedge \delta, \quad X_\delta^\varepsilon = (X\vee \varepsilon)\wedge \delta = (X \wedge \delta) \vee \varepsilon,
    \]
    and note that
    \[
     X^\varepsilon \downarrow X \textnormal{ as $\varepsilon \downarrow 0$,} \quad X_\delta \uparrow X \textnormal{ as $\delta \to \infty$,} \quad X_\delta^\varepsilon \downarrow X_\delta \textnormal{ as $\varepsilon \downarrow 0$,} \quad X_\delta^\varepsilon \uparrow X^\varepsilon \textnormal{ as $\delta \to \infty$.}
    \]
    The random variable $X_\delta^\varepsilon$ takes its values in the compact set $[\varepsilon, \delta]$, on which $f$ has a minimum and a maximum. Thus, both $X_\delta^\varepsilon$ and $f(X_\delta^\varepsilon)$ are bounded, and Jensen's inequality for integrable random variables yields
    \begin{equation} \label{eq:Jensen-bounded-approximation}
        \E[f(X_\delta^\varepsilon) \mid  \mathcal{G}] \leq f(\E[X_\delta^\varepsilon\mid  \mathcal{G}]) \quad \textnormal{a.s.}
    \end{equation}
    The desired result follows by letting $\varepsilon \downarrow 0$, and $\delta \to \infty$, in conjunction with appropriate limit theorems for conditional expectations. To that end, note that either one of the following three situations is the case:
    \begin{enumerate}
        \item $\lim_{x \to \infty}f(x) \in (-\infty, \infty]$;
        \item $\lim_{x \to \infty}f(x) = -\infty$, and $\lim_{x \downarrow 0}f(x) = -\infty$;
        \item $\lim_{x \to \infty}f(x) = -\infty$, and $\lim_{x\downarrow 0} f(x) \in \mathbb{R}$.
    \end{enumerate}
    We treat them separately.
    \begin{enumerate}
        \item In the first situation, the function $f$ is monotonically increasing. Thus, $f(X_\delta^\varepsilon)$ and $X_\delta^\varepsilon$ are bounded from below for fixed $\varepsilon$. Since $f(X_\delta^\varepsilon) \uparrow f(X^\varepsilon)$ and $X_\delta^\varepsilon \uparrow X^\varepsilon$ as $\delta \to \infty$, the monotone convergence theorem for (general) conditional expectations \citep[Chapter 2.7.4; Theorem 2]{shiryaev2016probability} yields
        \[
        \lim_{\delta \to \infty}\E[f(X_\delta^\varepsilon) \mid  \mathcal{G}] = \E[f(X^\varepsilon) \mid  \mathcal{G}], \quad \lim_{\delta \to \infty}f(\E[X_\delta^\varepsilon \mid  \mathcal{G}]) = f(\E[X^\varepsilon\mid  \mathcal{G}])\quad \textnormal{a.s.}
        \]

    By the monotonicity of $f$, it follows that for $0<\varepsilon < \eta$, $f(x \vee \varepsilon) \leq f(x \vee \eta)$, and by the concavity of $f$, there exists a constant $M(\eta) > 0$, such that $f(x \vee \eta)^+ \leq f(x)^+ + M(\eta)$ for all $x$. Thus, $f(X^\varepsilon) \leq f(X^\eta) \leq f(X)^+ + M(\eta)$, so the family $(f(X^\varepsilon) )_{0<\varepsilon \leq \eta}$ is bounded from above by an integrable random variable. Since $f(X^\varepsilon) \downarrow f(X)$ as $\varepsilon \downarrow 0$, the monotone convergence theorem again yields
    \[
        \lim_{\varepsilon \downarrow 0}\E[f(X^\varepsilon) \mid  \mathcal{G}] = \E[f(X) \mid  \mathcal{G}] \quad \textnormal{a.s.}
    \]
    Furthermore, by the monotonicity of $f$ and the fact that $X^\varepsilon \leq X+\varepsilon$, we have
    \[
    \lim_{\varepsilon \downarrow 0}f(\E[X^\varepsilon\mid  \mathcal{G}]) \leq \lim_{\varepsilon \downarrow 0}f(\E[X\mid  \mathcal{G}] + \varepsilon) = f(\E[X\mid  \mathcal{G}]) \quad \textnormal{a.s.}
    \]
    By combining all of these limit results with \eqref{eq:Jensen-bounded-approximation}, we find that
    \begin{align*}
    \E[f(X)\mid  \mathcal{G}] = \lim_{\varepsilon \downarrow 0}\lim_{\delta \to \infty} \E[f(X_\delta^\varepsilon) \mid  \mathcal{G}]\leq \lim_{\varepsilon \downarrow 0}\lim_{\delta \to \infty}f(\E[X_\delta^\varepsilon\mid  \mathcal{G}]) \leq f(\E[X\mid  \mathcal{G}]) \quad \textnormal{a.s,}
    \end{align*}
    as desired.
    \item In the second situation, $f$ is bounded from above and eventually decreases as $x \downarrow 0$ and as $x \to \infty$. That is, there exist constants $\eta_0 < \eta_1$, such that $f(x)$ is increasing on $(0, \eta_0)$ and decreasing on $(\eta_1, \infty)$. Take $\varepsilon < \eta_0 < \eta_1 < \delta$. Then $f(X^\varepsilon_\delta)$ is constant in $\varepsilon$ on the event $\{X_\delta \geq \eta_0\}$, while $f(X^\varepsilon_\delta) \downarrow f(X_\delta)$ as $\varepsilon \downarrow 0$ on the event $\{X_\delta < \eta_0\}$. Thus, by monotone convergence
    \begin{align*}
        \lim_{\varepsilon \downarrow 0}\E[ f(X^\varepsilon_\delta) \mid \mathcal{G}] = \E[ f(X_\delta) \mid \mathcal{G}], \quad \lim_{\varepsilon \downarrow 0}f(\E[ X^\varepsilon_\delta \mid \mathcal{G}]) = f(\E[ X_\delta \mid \mathcal{G}]) \quad \textnormal{a.s.}
    \end{align*}
    Similarly, $f(X_\delta)$ is constant in $\delta$ on the event $\{X \leq \eta_1\}$, while $f(X_\delta) \downarrow f(X)$ as $\delta \to \infty$ on the event $\{X > \eta_1\}$. Again, by monotone convergence,
    \begin{align*}
        \lim_{\delta \to \infty} \E[f(X_\delta) \mid \mathcal{G}] = \E[f(X) \mid \mathcal{G}], \quad \lim_{\delta \to \infty}f(\E[X_\delta \mid \mathcal{G}]) = f(\E[X \mid \mathcal{G}]) \quad \textnormal{a.s,}
    \end{align*}
    so that
    \begin{align} \label{eq:Jensen-limits-in-between}
        \E[f(X)\mid  \mathcal{G}] = \lim_{\delta \to \infty}\lim_{\varepsilon \downarrow 0} \E[f(X_\delta^\varepsilon) \mid  \mathcal{G}]\leq \lim_{\delta \to \infty}\lim_{\varepsilon \downarrow 0}f(\E[X_\delta^\varepsilon\mid  \mathcal{G}]) =f(\E[X\mid  \mathcal{G}]) \quad \textnormal{a.s.}
    \end{align}
    \item In the third situation, $f$ is again bounded from above, but it need not be true anymore that $f(x)$ eventually decreases as $x \downarrow 0$. However, because $f(x)$ converges to a finite limit as $x \downarrow 0$, we may conclude that $f$ has a minimum and a maximum on the domain $[0, \delta]$. Thus, for $\eta < \delta$, we have that the family $(f(X_\delta^\varepsilon))_{0 < \varepsilon \leq \eta}$ is uniformly bounded. This is also true for $(X_\delta^\varepsilon)_{0<\varepsilon \leq \eta}$. Whence, by the dominated convergence theorem \citep[Chapter 2.7.4; Theorem 2]{shiryaev2016probability},
    \[
    \lim_{\varepsilon \downarrow 0}\E[f(X_\delta^\varepsilon) \mid  \mathcal{G}] = \E[f(X_\delta) \mid  \mathcal{G}], \quad \lim_{\varepsilon \downarrow 0}f(\E[X_\delta^\varepsilon \mid  \mathcal{G}]) = f(\E[X_\delta \mid  \mathcal{G}]) \quad \textnormal{a.s.}
    \]
    Now, by letting $\delta \to \infty$ and applying the monotone convergence theorem in a similar fashion as in the second situation, we again obtain \eqref{eq:Jensen-limits-in-between}.
    \end{enumerate}

    Suppose, now, that $\E[f(X)^-] < \infty$. Then, the function $g_n(x) = f(x) \wedge n$ is bounded from above and concave. In particular, $\E[g_n(X)^+] < \infty$, so
    \[
    \E[g_n(X) \mid  \mathcal{G}] \leq g_n(\E[X \mid  \mathcal{G}]) \quad \textnormal{a.s.}
    \]
    Clearly, $g_n(x) \uparrow f(x)$ as $n \to \infty$, and because $g_n(X)$ is bounded from below by the integrable random variable $-f(X)^-$, the monotone convergence theorem yields
    \[
    \E[f(X) \mid  \mathcal{G}] = \lim_{n \to \infty}\E[g_n(X) \mid  \mathcal{G}] \leq \lim_{n\to \infty}g_n(\E[X \mid  \mathcal{G}]) = f(\E[X \mid  \mathcal{G}]) \quad \textnormal{a.s,}
    \]
    which is what we wanted to show.

    The final assertion follows from the tower property for general conditional expectations \citep[Chapter 2.7.4]{shiryaev2016probability}:
    \[
    \E[f(X)] = \E[\E[f(X) \mid  \mathcal{G}]] \leq \E[f(\E[X \mid  \mathcal{G}])],
    \]
    provided that the latter expectation is also defined.
\end{proof}

\section{Omitted proofs}
\label{sec:omitted-proofs}
\subsection{\texorpdfstring{Proof of Lemma~\ref{lem:log-utility-defined}}{Proof of Lemma 1}}
\label{proof:log-utility-defined}
By definition of KL-divergences, because $\inf_{P \in \Peff}\KL(P_\theta \| P)< \infty$, there is at least one element $P' \in \Peff$ that dominates $P_\theta$. In fact, one can, then, take $P'$ to be the reverse information projection (RIPr) as defined in \citet{larsson2025numeraire} of $P_\theta$ onto $\Peff$. This reverse information projection $P'$ is a sub-probability measure that additionally satisfies $P' \ll P_\theta$, and achieves the infimum $\KL(P_\theta \|  P') = \inf_{P \in \Peff}\KL(P_\theta \| P)$. (For further details, we refer to \citet{larsson2025numeraire}.)

Now, let $E$ be an e-variable, and define the (sub-) probability measure $R$ via the relation $dR = EdP'$. Then,
    \begin{align*}
    \E_{\theta}[\log(E)] = \E_{\theta}\left[ \log \frac{dR}{dP'}\right] &= \E_\theta\left[\log \frac{dP_\theta}{dP'} + \log\frac{dR}{dP_\theta}\right]
    = \KL(P_\theta \|  P') + \E_\theta\left[\log\frac{dR}{dP_\theta}\right].
    \end{align*}
    Since $\E_\theta[dR/dP_\theta] \leq 1$, Jensen's inequality implies that $\E_\theta[\log (dR/dP_\theta)] \in [-\infty, 0]$. This, combined with the fact that $\KL(P_\theta \|  P') < \infty$, implies that
    \[
    -\infty \leq \E_\theta[\log(E)] \leq \KL(P_\theta \| P') < \infty,
    \] and $\E_\theta[\log(E)]$ is therefore defined. \qed
\subsection{\texorpdfstring{Proof of Theorem~\ref{thm:rb-log-util-general}}{Proof of Theorem 2}}
\label{proof:rb-log-util-general}
Fix $\theta\in\Theta$. Because $E/G$ is nonnegative, the conditional expectation $\E_\theta[E/G \mid S]$ is defined. Then, $P_\theta$-almost surely,
\begin{align*}
    \E_\theta\left[\frac{E}{G} \given S\right] &=  \E_\theta\left[\frac{E}{G} \ind{\{G < \infty\}} \given S\right] +  \E_\theta\left[\frac{E}{G} \ind{\{G = \infty\}} \given S\right] \\
    &\leq \E_\theta\left[\frac{E}{G} \ind{\{G < \infty\}} \given S\right] + \ind{\{G = \infty\}},
\end{align*}
because the conventions $0/0 = 0$ and $\infty/\infty = 1$ ensure that $E/G \leq 1$ on the $\sigma(S)$-measurable set $\{G = \infty\}$. Now, as $G = \E_\theta[E \mid S]$ $P_\theta$-a.s, it follows that $E = 0$ $P_\theta$-a.s on the set $\{G = 0\}$. Thus, we may write
\[
    \frac{E}{G} \ind{\{G < \infty\}} = \left(\frac{\ind{\{0<G < \infty\}}}{G}\right) E \quad P_\theta\textnormal{-a.s,}
\]
so that the product rule for (general) conditional expectations implies that $P_\theta$-a.s,
\begin{align*}
\E_\theta\left[ \frac{E}{G} \ind{\{G < \infty\}}\given S\right] &= \left(\frac{\ind{\{0<G < \infty\}}}{G}\right) \E_\theta [E \mid S] \\
&= \left(\frac{\ind{\{0<G < \infty\}}}{G}\right) G = \ind{\{0<G < \infty\}}.
\end{align*}
Whence,
\[
\E_\theta\left[\frac{E}{G} \given S\right] \leq \ind{\{0<G < \infty\}} + \ind{\{G = \infty\}} \leq 1 \quad P_\theta\textnormal{-a.s}.
\]
The inequalities $\E_\theta\left[E/G \right] \leq 1$ and  $\E_\theta\left[\log(E/G)\right] \leq 0$
then follow by the tower property and Jensen's inequality, completing the proof. \qed

\end{document}